\documentclass[submission]{eptcs}
\providecommand{\event}{ACT 2020} 
\usepackage{underscore}    
\usepackage[utf8]{inputenc}
\usepackage{fancyhdr}
\usepackage{amsmath,amssymb,amsthm}
\usepackage{tikz}
\usepackage{wrapfig} 
\usepackage{graphicx} 
\usepackage{tikz-cd}
\usepackage[backend=biber]{biblatex}
\usepackage{tipa} 
\usepackage{mathtools} 
\usepackage{caption}
\usepackage{thmtools, thm-restate} 
\usepackage{array,multirow}

\usetikzlibrary{decorations.pathreplacing} 

\graphicspath{graphics/}

\DeclareMathOperator{\Aa}{\mathcal{A}}
\DeclareMathOperator{\Ba}{\mathcal{B}}
\DeclareMathOperator{\Ca}{\mathcal{C}}
\DeclareMathOperator{\Da}{\mathcal{D}}
\DeclareMathOperator{\Ea}{\mathcal{E}}

\DeclareMathOperator{\Ib}{\mathbb{I}}

\DeclareMathOperator{\Rb}{\mathbb{R}}

\DeclareMathOperator{\op}{^\text{op}}

\DeclareMathOperator{\id}{\mathsf{id}}

\DeclareMathOperator{\Set}{\textbf{Set}}

\DeclareMathOperator{\Cat}{\textbf{Cat}}

\newcommand{\type}[1]{\mathsf{{#1}}}

\newcommand{\lens}[2]{\begin{pmatrix}{#1} \\ {#2} \end{pmatrix}}

\newcommand{\xto}[1]{\xrightarrow{#1}}

\newcommand{\pto}{\,\cdot\kern-.1em{\to}\,}

\makeatletter
\providecommand*{\xmapstofill@}{%
  \arrowfill@{\mapstochar\relbar}\relbar\rightarrow
}
\providecommand*{\xmapsto}[2][]{%
  \ext@arrow 0395\xmapstofill@{#1}{#2}%
}

\makeatletter
\def\slashedarrowfill@#1#2#3#4#5{%
  $\m@th\thickmuskip0mu\medmuskip\thickmuskip\thinmuskip\thickmuskip
   \relax#5#1\mkern-7mu%
   \cleaders\hbox{$#5\mkern-2mu#2\mkern-2mu$}\hfill
   \mathclap{#3}\mathclap{#2}%
   \cleaders\hbox{$#5\mkern-2mu#2\mkern-2mu$}\hfill
   \mkern-7mu#4$%
}
\def\rightslashedarrowfill@{%
  \slashedarrowfill@\relbar\relbar\mapstochar\rightarrow}
\newcommand\xslashedrightarrow[2][]{%
  \ext@arrow 0055{\rightslashedarrowfill@}{#1}{#2}}
\makeatother

\newcommand{\topro}{\xslashedrightarrow{}}

\tikzset{
    vert/.style={anchor=south, rotate=90, inner sep=.5mm}
} 

\newtheorem{thm}{Theorem}[section]

\theoremstyle{definition}
\newtheorem{defn}[thm]{Definition}

\newtheorem{rmk}[thm]{Remark}

\newtheorem*{acknowledgements}{Acknowledgements}

\newtheorem{cor}[thm]{Corollary}
\newtheorem{prop}[thm]{Proposition}

\title{Double Categories of Open Dynamical Systems (Extended Abstract)}
\author{David Jaz Myers
\institute{Johns Hopkins University}}

\def\titlerunning{Double Categories of Open Dynamical Systems}
\def\authorrunning{David Jaz Myers}

\begin{document}

\maketitle

\begin{abstract}
   A (closed) dynamical system is a notion of how things can be, together with a notion of how they may change given how they are. The idea and mathematics of closed dynamical systems has proven incredibly useful in those sciences that can isolate their object of study from its environment. But many changing situations in the world cannot be meaningfully isolated from their environment -- a cell will die if it is removed from everything beyond its walls. To study systems that interact with their environment, and to design such systems in a modular way, we need a robust theory of open dynamical systems. 

   In this extended abstract, we put forward a general definition of open
   dynamical system. We define two general sorts of morphisms between these systems: covariant morphisms which include trajectories, steady states, and periodic orbits; and contravariant morphisms which allow for
   plugging variables of some systems into parameters of other systems. We
   define an indexed double category of open dynamical systems indexed by their
   interface and use a
   double Grothendieck construction to construct a double category
   of open dynamical systems.

   In our main theorem, we construct covariantly representable indexed double functors
   from the indexed double category of dynamical systems to an indexed double
   category of spans. This shows that all covariantly representable structures
   of dynamical systems --- including trajectories, steady states, and periodic
   orbits --- compose according to the laws of matrix arithmetic.
\end{abstract}

\section{Open Dynamical Systems}

The notion of a \emph{dynamical system} pervades mathematical modeling in the
sciences. A dynamical system consists of a way things might be, and a way things might change given how they are. There are many
doctrines in which this notion may be interpreted:

\begin{itemize}
  \item In a \emph{discrete} dynamical system, we have a set $S$ of states and
    an update function $u : S \to S$ which assigns to the current state $s \in
    S$ of the system the next state $u(s) \in S$.
  \item In a \emph{Markov model}, we have an $n$-element set $S$ of states, and
    an $n \times n$ stochastic matrix $U$ whose $(i, j)$ entry $U_{ij}$ is the
    probability that state $i$ will transition to state $j$. We can also see
    this as a function $u : S \to DS$, where $DS$ is the set of probability
    distributions on $S$ by the relation $u(i)_j := U_{ij}$. 
  \item A \emph{continuous-time} dynamical system is often given by a system of
    differential equations. We have a manifold $S$ of states (often $\Rb^n$, where
    $n$ is the number of state variables), and a vector field $u : S \to TS$
    giving the differential equation
    $$\frac{ds}{dt} = u(s).$$
\end{itemize}

These systems are all \emph{closed} in the sense that they do not depend on
external parameters. However, real-world systems are seldom closed, and our
models of them often depend on external parameters. Furthermore, these
parameters may themselves depend on certain variables which are exposed by other
dynamical systems.

An \emph{open dynamical system} is a system whose dynamics
may depend on external parameters (which we will call \emph{inputs}) and which exposes some variables of its
states (which we will call \emph{outputs}). The above examples of closed dynamical systems have open analogues:
\begin{itemize}
  \item A \emph{deterministic automaton} consists of an input alphabet $I$, an
    output alphabet $O$, a set of states $S$, a readout function $r : S \to O$
    which exposes the output symbol $r(s)$ of the state $s$, and an update
    function $u : S \times I \to S$ which takes the current state $s$ and an
    input symbol $i$ and yields the state $u(s, i)$ that the automaton will
    transition into when reading $i$ in state $s$.
  \item A \emph{Markov decision process} consists of a set $S$ of states, a set
    $O$ of orientations the agent may take in the environment, a set of actions $I$, a readout
    function $r : S \to O$ which extracts the orientation of the agent in a
    given state, and a stochastic
    update function $u : S \times I \to DS$ which, for every action $i$ and
    state $s$ gives a probability distribution $u(s, i)$ on the states of $S$
    representing the likely transitions of the system given that action $i$ is
    taken in state $s$.
    Often one includes an expected reward, so that $u$ instead has
    signature $S \times I \to D(\Rb \times S)$. 
  \item An \emph{open continuous-time} dynamical system (see \cite{schultz2016dynamical}) corresponds to a family
    of differential equations
$$\frac{ds}{dt} = u(s, i)$$
    concerning a variable state $s \in S$ varying with
    a choice of parameter $i \in I$, and exposing a variable $r(s) \in O$.
\end{itemize}

These various sorts of dynamical systems have in common the following general form:
\begin{itemize}
\item They involve a notion
of state space $S$ which takes place in a category that also contains their
output space $O$ so that the readout $r : S \to O$ can be a morphism in this
category. In the examples of deterministic automata and Markov decision
processes, the state and output spaces are sets; for a continuous-time dynamical system, they are
differentiable manifolds. We will refer to these spaces in general as
\emph{contexts}, and so we will begin with a category $\Ca$ of contexts.
\item They involve a notion of bundle over the state space, or of contextualized
  maps between contexts. That is, to every context $C$, there is a category
  $\type{Bun}(C)$ of \emph{actions} possible in the context $C$. We see that not only is the space of possible changes in a given state $TS \in
  \type{Bun}(S)$ a bundle in this sense, but also the inputs $I \in
  \type{Bun}(O)$ are a bundle over the outputs.

  In order for the update function to map from $I$ to $TS$ in $\type{Bun}(S)$,
  we must be able to recontextualize (or pull back) bundles along maps of contexts. Therefore,
  we will ask that $\type{Bun} : \Ca\op \to \Cat$ be an indexed category. We see
  then that we can pull back the inputs $I$ along $r : S \to O$ so that the update
  $u$ may have signature $u : r^{\ast} I \to TS$. 
\item For every context $S$, we must have a canonical bundle $TS \in \type{Bun}(S)$ of
  changes possible in each state. We want $T$ to covary with states,
  so that we may pushforward changes alongs maps between state spaces.
  Therefore, we ask that $T$ be a section of the indexed category $\type{Bun} :
  \Ca\op \to \Cat$
\end{itemize}

We refer to the data of an indexed category $\type{Bun} : \Ca\op \to \Ca$ with a
section $T$ collectively as a \emph{dynamical system doctrine}, or
just \emph{doctrine}.
\begin{defn}
A \emph{dynamical system doctrine} is an indexed category $\type{Bun} : \Ca\op
\to \Cat$ with a section $T$ of its Grothendieck construction.
\end{defn}

Motivated by the examples above, we define a $(\type{Bun}, T)$-dynamical system to consist of:
  \begin{itemize}
  \item A space $S \in \Ca$ of states.
    \item A space $O \in \Ca$ of outputs or \emph{orientations}.
    \item A bundle $I \in \type{Bun}(O)$ giving the inputs or parameters valid
      in a given orientation.
    \item A readout map $r : S \to O$ extracting the orientation of the system
      in a given state.
    \item An update function $r^{\ast} I \to TS$ in $\type{Bun}(S)$ which sends
      each input valid in a given state to the resulting change in the system in
      $\type{Bun}(S)$.
  \end{itemize}

The above open dynamical systems arise for various choices of doctrine. Namely:
\begin{itemize}
  \item Deterministic automata arise by taking $\type{Bun}(C):=
    \textbf{CoKleisli}(C \times -)$, the coKleisli category for the comonad $C
    \times -$, together with the section $C \mapsto C$. We will refer to this as
    the \emph{deterministic doctrine}.
  \item Markov decision process arise by taking $\type{Bun}(C) :=
    \textbf{BiKleisli}(C\times -, D)$, the biKleisli category of the comonad $C
    \times -$
    distributing over the strong monad of probability distributions $D$. Any
    strong monad will work here\footnote{A commutative monad is necessary for
      the monoidal structure.}, for example, the monad $D(\Rb \times -)$ which
    keeps track of an expected $\Rb$-valued reward, or the powerset monad which
    allows for non-determinism. We will refer to this as the \emph{monadic
    doctrine}.
  \item Continous-time dynamical systems arise by taking $\type{Bun}(C) :=
    \type{Subm}(C)$ to be the category of submersions $M \to C$, with section
    $T$ given by taking the tangent bundle. We will refer to this as the
    \emph{continuous doctrine}.
\end{itemize}

\begin{acknowledgements}
The author would like to thank David Spivak, Emily Riehl, and Sophie Libkind for fruitful
conversation and comments during drafting. The author also appreciates support
from the National Science Foundation grant DMS-1652600.
\end{acknowledgements}
  
\section{Contravariant Morphisms: Plugging Variables into Parameters}

We may plug the variables exposed by one dynamical system into the parameters of
other dynamical systems to create a more complex dynamical system. For example,
consider a rabbit population $r$ which reproduces at a rate $\alpha$ and is
eaten by a predator at a rate $\beta$:
$$\frac{dr}{dt} = \alpha  r - \beta r.$$
If we have a population of foxes $f$ which reproduce at a rate $\gamma$ and die
at a rate $\delta$:
$$\frac{df}{dt} = \gamma f - \delta f,$$
we may want to say that the rate $\beta$ at which rabbits are eaten is
proportional to the population of foxes, and the rate at which foxes breed
depends on how many rabbits they eat:
\begin{align*}
  \beta &= cf \\
  \gamma &= dr
\end{align*}
Making this substitution, we get the final system of equations, usually known as
the ``Lotka-Volterra predator-prey model'':
\begin{equation}\label{eqn:lotka.volterra}
  \begin{aligned}
  \frac{dr}{dt} &= \alpha r - cfr \\
  \frac{df}{dt} &= drf - \delta f.
  \end{aligned}
\end{equation}

This sort of ``plugging in'' of the exposed variables of one system into the
parameters of another is governed by \emph{lens composition} in general. In
special cases, lens composition can be described by an algebra of wiring
diagrams \cite{schultz2016dynamical}.

In the general setting of an indexed category $\type{Bun} : \Ca\op \to \Cat$, we use the notion of a \emph{generalized lens} due to Spivak
\cite{spivak2019generalized} to govern this sort of ``plugging in'' operation.
\begin{defn}\label{defn:generalized.lens}
Given an indexed category $\type{Bun} : \Ca\op \to \Cat$, the category of
$\type{Bun}$-lenses is the contravariant Grothendieck construction of the
pointwise opposite of $\type{Bun}$.
$$\textbf{Lens}_{\type{Bun}} := \int_{C : \Ca} \type{Bun}(C)\op. $$

We denote an object of the category of lenses by $\lens{A}{C}$ where $C \in \Ca$
and $A \in \type{Bun}(C)$, and we write a morphism in the category of lenses
(itself called a lens) as
$$\lens{f^{\sharp}}{f} : \lens{A}{C} \leftrightarrows \lens{A'}{C'}$$
where $f : C \to C'$ and $f^{\sharp} : f^{\ast}A' \to A$.
\end{defn}

We can see open dynamical systems as particular sorts of generalized lenses,
and they may therefore be acted upon by generalized lenses via composition.
The data of a $(\type{Bun}, T)$-dynamical system can be described as a
$\type{Bun}$-lens
$$\lens{u}{r} : \lens{TS}{S} \leftrightarrows \lens{I}{O}.$$
Therefore, given any lens $\lens{f^{\sharp}}{f} : \lens{I}{O} \leftrightarrows
\lens{I'}{O'}$, we may compose to get a new dynamical system
$$\lens{f^{\sharp}}{f} \circ \lens{u}{r} : \lens{TS}{S} \leftrightarrows \lens{I'}{O'}.$$
In particular, we can formalize the Lotka-Volterra system as follows. The rabbit system
may be described as having
\begin{itemize}
\item state space $\Rb$,
\item output space $\Rb$, with readout $\id : \Rb \to \Rb$,
\item input bundle $\Rb \times \Rb^2 \to \Rb$, and
\item update $u : \Rb \times \Rb^2 \to T\Rb$ given by
  $$u(r, (\alpha, \beta)) := (\alpha r - \beta r) \frac{d}{dr}$$.
\end{itemize}
The fox system is defined in the same way. We may then form the product system
by taking the cartesian product of all the data involved.\footnote{
Though for reasons of space we will not dwell on the issue in this extended abstract, all of
our constructions can be carried out on a \emph{monoidal} indexed category with
a pseudo-monoidal section,
yielding a monoidal structure on the double category of open dynamical systems.
This monoidal structure is important: it is the way we combine systems
disjointly in
preparation to plug in the variables of some into the parameters of others. Our
covariantly representable indexed double functors of Section
\ref{sec:representable.functors} are also lax monoidal with repespect to this
monoidal structure.
} This gives us, in
total, a generalized lens
$$\lens{T(\Rb \times \Rb)}{\Rb \times \Rb} \leftrightarrows \lens{(\Rb \times
  \Rb) \times (\Rb^2 \times \Rb^2)}{\Rb \times \Rb}$$
We now consider the lens $\lens{e^{\sharp}}{e} : \lens{(\Rb \times
  \Rb) \times (\Rb^2 \times \Rb^2)}{\Rb \times \Rb} \leftrightarrows \lens{(\Rb
  \times \Rb) \times (\Rb^2 \times \Rb^2)}{\Rb \times \Rb}$ given by
\begin{align*}
  e(r, f) &= (r, f) \\
  e^{\sharp}((r, f), ((\alpha, c), (d, \delta))) &= ((r, f), ((\alpha, cf), (dr, \delta)))
\end{align*}
The composite
$$\lens{T(\Rb \times \Rb)}{\Rb \times \Rb} \leftrightarrows \lens{(\Rb \times
  \Rb) \times (\Rb^2 \times \Rb^2)}{\Rb \times \Rb} \leftrightarrows \lens{(\Rb \times
  \Rb) \times (\Rb^2 \times \Rb^2)}{\Rb \times \Rb}$$
is then the combined system of Eqn \ref{eqn:lotka.volterra}. This general way of
combining open continuous-time dynamical systems was explored in
\cite{schultz2016dynamical}. These contravariant morphisms will form the
vertical morphisms in our double categories of dynamical systems.

We will see this system of ``plugging in'' equations via lens composition as a
\emph{contravariant morphism} between open dynamical systems. Other examples of
contravariant morphisms of open dynamical systems include the cascade products
of automata important for Kohn-Rhodes theory \cite{Krohn1965}\footnote{The author would like to thank Sophie Libkind for pointing out the
relationship between contravariant morphisms and cascade products}, and hierarchical
planning schemes for Markov decision processes \cite{spivakmarkov}. 

\section{Covariant Morphisms: Trajectories, Steady States, and Periodic Orbits}\label{sec:covariant.morphisms}

In addition to contravariant morphisms of open dynamical systems, there are also the
\emph{covariant morphisms} by which one system is directly mapped onto another.
These include trajectories, steady states, and periodic orbits. For example, consider the
continuous-time dynamical system
$$\lens{\frac{d}{ds}}{\id} : \lens{T\Rb}{\Rb} \leftrightarrows \lens{\Rb}{\Rb}.$$
This system represents the very simple differential equation
$$\frac{ds}{dt} = 1.$$
Despite its simplicity, it is of crucial importance for the study of continuous-time differential
equations because of what it represents: the notion of trajectory. It is the ``walking trajectory''. Namely, let $\lens{u}{\id} :
\lens{TS}{S} \leftrightarrows \lens{S \times I}{S}$ be
the Lotka-Volterra model, and consider a smooth map $(r, f) : \Rb \to
S$ together with a bundle morphism
\[
\begin{tikzcd}[column sep = large]
\Rb \arrow[d] \arrow[r, "{((\alpha, c), (d, \delta))}"] & S \times I \arrow[d] \\
\Rb \arrow[r, "{(r, f)}"']                              & S                   
\end{tikzcd}
\]
subject to the relation that
\begin{align*}
u((r, f), ((\alpha, c), (d, \delta))) = \frac{dr}{dt}\frac{d}{dr} + \frac{df}{dt} \frac{d}{df}.
\end{align*}
Or, in other words, functions $r, f, \alpha, c, d, \delta : \Rb \to \Rb$ so that
Eqn \ref{eqn:lotka.volterra} is satisfied for all $t \in \Rb$:
\begin{align*}
  \frac{dr}{dt}(t) &= \alpha(t) r(t) - c(t)f(t)r(t) \\
  \frac{df}{dt}(t) &= d(t)r(t)f(t) - \delta(t) f(t).
\end{align*}
This is a solution to the system of equations, given a choice
of parameter for all times $t$.

Another example of a covariant morphism is a steady state. We recall from
\cite{spivak2015steady} the definition of steady state for a deterministic
automaton:
\begin{defn}[Definition 2.4 of \cite{spivak2015steady}]
Let $\lens{u}{r} : \lens{S}{S} \leftrightarrows \lens{I}{O}$ be a deterministic
automaton. For $o \in O$ and $i \in I$, an $(i, o)$-steady state is a state $s
\in S$ such that $r(s) = o$ and $u(s, i) = s$.
\end{defn}
Consider the trivial deterministic automaton $\lens{\id}{\id} :
\lens{\ast}{\ast} \leftrightarrows \lens{\ast}{\ast}$. Note that we may see a
state $s \in S$ of a deterministic automaton $\lens{u}{r}$ as a map from the
states of $\lens{\id}{\id}$ to the states of $\lens{u}{r}$, and similarly an output $o \in O$
and an input $i \in I$ as maps from the outputs and inputs of $\lens{\id}{\id}$
respectively. If we require these to satisfy the laws:
\begin{align*}
  r(s) &= o \\
  u(s, i) &= s
\end{align*}
then we will have a steady state of $\lens{u}{r}$.

These sorts of maps are instances of covariant
morphisms of dynamical systems. The general relations that such morphisms must
satisfy involve commutation between lenses and bundle maps. We will therefore
work with a \emph{double category} whose vertical morphisms are lenses and whose
horizontal morphisms are bundle maps. The construction of this double category
is quite general; we refer to it as the \emph{Grothendieck double
  construction}.

\section{The Double Category of Interfaces}\label{sec:groth.double.construction}

We can make a double category whose vertical morphisms are lenses and whose
horizontal morphisms are bundle morphisms. We call this construction the
\emph{Grothendieck double construction}, and refer to the resulting double
category as the double category of \emph{interfaces}.

\begin{defn}\label{defn:groth.double.construction}
Let $\type{Bun} : \Ca\op \to \Cat$ be an indexed category. Its \emph{Grothendieck double construction}, which we
will refer to as the double category $\type{Interface}$ of \emph{interfaces}, is the double category with:
\begin{itemize}
    \item Objects pairs $\lens{A}{C}$ with $C \in \Ca$ and $A \in \type{Bun}(C)$.
    \item Vertical morphisms $\lens{f^{\sharp}}{f} : \lens{A}{C}
      \leftrightarrows \lens{A'}{C'}$ are $\type{Bun}$-lenses, that is, morphisms in the Grothendieck construction $\int \type{Bun}\op$ of the pointwise opposite of $\type{Bun}$, namely pairs $f : C \to C'$ and $f^{\sharp} : f^{\ast}A' \to A$.
    \item Horizontal morphisms $\lens{g_{\sharp}}{g} : \lens{A}{C}
      \rightrightarrows \lens{A'}{C'}$ are $\type{Bun}$-maps, that is, morphisms in the Grothendick construction $\int \type{Bun}$ of $\type{Bun}$, namely pairs $g : C \to C'$ and $g_{\sharp} : A \to g^{\ast} A'$.
    \item There is a square
    \[
    \begin{tikzcd}
        \lens{A_1}{C_1} \arrow[r, shift left, "\lens{g_{1\sharp}}{g_1}"]\arrow[r, shift right] \arrow[d, leftarrow,  shift left] \arrow[d, shift right, "\lens{f_1^{\sharp}}{f_1}"'] & \lens{A_2}{C_2} \arrow[d, leftarrow, shift left, "\lens{f_2^{\sharp}}{f_2}"] \arrow[d, shift right] \\
        \lens{A_3}{C_3} \arrow[r, shift left]\arrow[r, shift right, "\lens{g_{2\sharp}}{g_2}"'] & \lens{A_4}{C_4}
    \end{tikzcd}
    \]
    if and only if the following diagrams commute:
    \[
\begin{tikzcd}
C_1 \arrow[r, "g_1"] \arrow[d, "f_1"'] & C_2 \arrow[d, "f_2"] &  & f_1^{\ast}A_3 \arrow[rr, "f_1^{\sharp}"] \arrow[d, "f_1^{\ast}g_{2\sharp}"'] &                                                              & A_1 \arrow[d, "g_{1\sharp}"] \\
C_3 \arrow[r, "g_2"']                  & C_4                  &  &
f_1^{\ast}g_2^{\ast}A_4 \arrow[r, equals]                                            & g_1^{\ast}f_2^{\ast}A_4 \arrow[r, "g_1^{\ast}f_2^{\sharp}"'] & g_1^{\ast}A_2               
\end{tikzcd}
    \]
We will call the squares in the Grothendieck double construction \emph{commuting squares}, since they represent the proposition that the ``lower'' and ``upper'' squares appearing in their boundary commute.
\end{itemize}

\end{defn}

Note that the horizontal category of the double category of interfaces is the
category of bundles and bundle maps --- the Grothendieck construction of $\type{Bun}$ --- while
the vertical category is the category of lenses --- the Grothendieck
construction of the pointwise opposite $\type{Bun}(-)\op$.

We can in fact express the double category of interfaces solely in terms of the
vertical/cartesian factorization system on the category $\int \type{Bun}$ of
bundles. This perspective is taken in the author's forthcoming
\cite{jaz2020cartesian}.
\begin{prop}[\cite{jaz2020cartesian}]\label{prop:groth.double.span}
Let $\type{Bun} : \Ca\op \to \Cat$ be an indexed category. The Grothendieck
double construction of $\type{Bun}$ is equivalent to the double category defined
by:
\begin{itemize}
  \item Its horizontal category is the Grothendieck construction $\int
    \type{Bun}$ of $\type{Bun}$. 
  \item A vertical morphism is a span
    \[
      \begin{tikzcd}
        & \lens{f^{\ast}A'}{C} \arrow[dr, equals, shift left, "\lens{\id}{f}"] \arrow[dr, shift right]
        \arrow[dl, equals, shift left] \arrow[dl, shift right, "\lens{f^{\sharp}}{\id}"']& \\
        \lens{A}{C} & & \lens{A'}{C'}
      \end{tikzcd}
    \]
    whose left leg is vertical and whose right leg is cartesian. These are
    composed by pullback in the usual way.
  \item A square is a map of spans in $\int \type{Bun}$ in the usual sense.
\end{itemize}
\end{prop}

\section{The Indexed Double Category of Dynamical Systems}\label{sec:dynamical.systems}

In this section, we define the indexed double category of $(\type{Bun}, T)$-dynamical systems associated to the dynamical system doctrine $(\type{Bun}, T)$
given by an indexed category $\type{Bun} : \Ca\op \to
\Cat$ equipped with a section $T : \Ca \to \int \type{Bun}$.

\begin{defn}
A (covariantly) \emph{indexed double category} is a lax double functor $F : \Da
\to \Cat$ from a double category $\Da$ to the double category of categories,
functors (vertical), and profunctors (horizontal). 
\end{defn}

\begin{defn}\label{defn:indexed.double.cat.of.dynamical.systems}
The indexed double category of $(\type{Bun}, T)$-dynamical systems $\type{Dyn} :
\type{Interface} \to \Cat$ is the lax double functor acting as:
\begin{itemize}
  \item $\type{Dyn}\lens{I}{O}$ is the category of $\lens{I}{O}$-dynamical
    systems. Namely, this is the category whose objects are dynamical systems
    $$\lens{u}{r} : \lens{TS}{S} \leftrightarrows \lens{I}{O}$$
    with morphisms given by squares
    \[
    \begin{tikzcd}
        \lens{TS}{S} \arrow[r, shift left, "\lens{T\varphi}{\varphi}"]\arrow[r, shift right] \arrow[d, leftarrow,  shift left] \arrow[d, shift right, "\lens{u}{r}"'] & \lens{TS'}{S'} \arrow[d, leftarrow, shift left, "\lens{u'}{r'}"] \arrow[d, shift right] \\
        \lens{I}{O} \arrow[r, shift right, equals]\arrow[r, shift left, equals] & \lens{I}{O}
    \end{tikzcd}
    \]
    and composed via horizontal composition in the double category of interfaces.
  \item A lens (vertical morphism) $\lens{f^{\sharp}}{f} : \lens{I}{O}
    \leftrightarrows \lens{I'}{O'}$ gives a functor $\type{Dyn}\lens{I}{O} \to
    \type{Dyn}\lens{I'}{O'}$ by vertical composition:
    \[
    \begin{tikzcd}
        \lens{TS}{S} \arrow[r, shift left, "\lens{T\varphi}{\varphi}"]\arrow[r, shift right] \arrow[d, leftarrow,  shift left] \arrow[d, shift right, "\lens{u}{r}"'] & \lens{TS'}{S'} \arrow[d, leftarrow, shift left, "\lens{u'}{r'}"] \arrow[d, shift right] \\
        \lens{I}{O} \arrow[r, shift right, equals]\arrow[r, shift left, equals]  \arrow[d, leftarrow,  shift left] \arrow[d, shift right, "\lens{f^{\sharp}}{f}"']
        & \lens{I}{O} \arrow[d, leftarrow, shift left, "\lens{f^{\sharp}}{f}"] \arrow[d, shift right]\\
        \lens{I'}{O'} \arrow[r, shift right, equals]\arrow[r, shift left, equals] & \lens{I'}{O'}
    \end{tikzcd}
    \]
  \item A bundle map $\lens{g_{\sharp}}{g} : \lens{I}{O} \rightrightarrows
    \lens{I'}{O'}$ gives a profunctor $\type{Dyn}\lens{I}{O}\op \times
    \type{Dyn}\lens{I'}{O'} \to \Set$ sending dynamical systems $\lens{u}{r} :
    \lens{TS}{S} \leftrightarrows \lens{I}{O}$ and $\lens{u'}{r'} :
    \lens{TS'}{S'} \leftrightarrows \lens{I'}{O'}$ to the set of squares:
\[
\left\{ 
    \begin{tikzcd}
        \lens{TS}{S} \arrow[r, dashed, shift left,
        "\lens{T\varphi}{\varphi}"]\arrow[r, dashed, shift right] \arrow[d, leftarrow,  shift left] \arrow[d, shift right, "\lens{u}{r}"'] & \lens{TS'}{S'} \arrow[d, leftarrow, shift left, "\lens{u'}{r'}"] \arrow[d, shift right] \\
        \lens{I}{O}\arrow[r, shift right, "\lens{g_{\sharp}}{g}"']\arrow[r, shift left] & \lens{I'}{O'}
    \end{tikzcd}
 \right\}
\]
and acting on the left and right via horizontal composition. The laxator and
unitor are given by horizontal composition and identity respectively.
\item A square gets sent to the morphism of profunctors given by composing with
  that square
\end{itemize}
\end{defn}

We may take a \emph{double Grothendieck construction} (not to be confused with
the earlier Grothendieck double construction) to get the double category of open
dynamical systems with variable interface. 

\begin{defn}\label{defn:double.groth.construction}
Let $\Ea : \Da \to \Cat$ be an indexed double
category. The \emph{double Grothendieck construction} $\iint \Ea$ is the
double category with:
\begin{itemize}
\item Objects pairs $(D, A)$ with $D \in \Da$ and $A \in \Ea(D)$.
\item Vertical morphisms pairs $(f, f_{\sharp}) : (D, A) \to (D', A')$ with $f :
  D \to D'$ vertical in $\Da$ and $f_{\sharp} : f_{\ast}(A) \to A'$ in
  $\Ea(D')$.
\item Horizontal morphisms $(g, g^{\sharp}) : (D, A) \topro (D', A')$ are pairs
  $g : D \topro D'$ and $g^{\sharp} \in \Ea(g)(A, A')$.
\item Squares
  \[
    \begin{tikzcd}
      (D, A) \arrow[r, "{ (g, g^{\sharp}) }"] \arrow[d, "{ (f, f_{\sharp}) }"'] \arrow[dr, phantom, "\alpha"] & (D', A') \arrow[d, "{ (f', f'_{\sharp}) }"]\\
      (D'', A'') \arrow[r, "{ (g', g'^{\sharp}) }"'] & (D''', A''')
    \end{tikzcd}
  \]
  is a square
  \[
    \begin{tikzcd}
      D \arrow[r, "g"] \arrow[d, "f"'] \arrow[dr, phantom, "\alpha"] & D'\arrow[d, "f'"]\\
      D'' \arrow[r, "g'"'] & D'''
    \end{tikzcd}
  \]
  such that $\Ea(\alpha)(g^{\sharp}) \cdot f'_{\sharp} = f_a \cdot g^{'\sharp}$.
\end{itemize}
Composition is given as follows:
\begin{itemize}
\item Vertical composition is given by
  $$(f', f'_{\sharp}) \circ (f, f_{\sharp}) \coloneqq  (f' \circ f, f'_{\sharp} \circ
  \Ea(f)(f_{\sharp})).$$
  The vertical identities are $(\id, \id)$.
\item Horizontal composition is given by
  $$(g', g^{'\sharp}) \circ (g, g^{\sharp}) \coloneqq (g' \circ g, \mu(g'^{\sharp},
  g^{\sharp}))$$
  where $\mu : \Ea(g') \otimes \Ea(g) \to \Ea(g' \circ g)$ is the laxator.

  The horizontal identities are $(\id, \eta(\id))$, where $\eta$ is the unitor.
\item Vertical and horizontal composition of squares are given as in $\Da$.
\end{itemize}
\end{defn}

\section{Indexed Double Functors Covariantly Represented by Dynamical Systems}\label{sec:representable.functors}

In this section, we consider functors out of the indexed double category of
dynamical systems. In particular, we will focus on covariantly representable
functors. As we saw in Section \ref{sec:covariant.morphisms}, notions such as
trajectories, steady states, and periodic orbits of open dynamical systems are
represented by covariant morphisms out of simple systems. In this section, we will construct covariantly representable indexed double
functors into an indexed double category of spans (Definition
\ref{defn:indexed.double.span}).

In \emph{The steady states of coupled dynamical systems compose via matrix
  arithmetic} \cite{spivak2015steady}, David Spivak shows that taking steady states of
deterministic automata (and other systems) gives a functor of wiring diagram
algebras into an wiring diagram algebra of matrices. In Theorem 4.40 of that
paper, Spivak shows that this functor factors through a wiring
diagram algebra of ``matrices of sets''\footnote{See Definition 4.34 of \emph{ibid.}}.

Our indexed double category of spans generalizes this algebra of matrices of sets
and therefore also the algebra of matrices. We can see a span $V
\xleftarrow{s} X \xto{t} W$ as a $V \times W$ matrix of sets $X_{vw}$ (the
fibers over $v \in V$ and $w \in W$), and composition of
spans as matrix multiplication. As a result, we can see Theorem
\ref{thm:representable.indexed.double.functor} as a generalization of Spivak's
\cite[Theorem 4.40]{spivak2015steady}, showing that any covariantly
representable structure associated to an open dynamical system --- including not only
steady states but also trajectories and periodic orbits --- composes via matrix arithmetic.

We begin by defining the codomain of these covariantly representable functors.
\begin{defn}\label{defn:indexed.double.span}
Let $\Aa$ be a category with finite limits, and consider the double categoy
$\type{Span}(\Aa)$ as with vertical morphisms the spans in $\Aa$ and horizontal
morphisms the maps in $\Aa$. Note that it is the vertical morphisms of
$\type{Span}(\Aa)$ which are the spans, contradicting a common (but not
universal) convention. We define the \emph{slice indexed double category}
$\Aa_{/(-)} : \type{Span}(\Aa) \to \Cat$ to be given by the following
assignments:
\begin{itemize}
  \item To every object $X$, $\Aa_{/X}$ is the slice category of $\Aa$ over $X$.
    Note that this is equivalently the category of vertical morphisms
    $v\type{Span}(\Aa)(\ast, \Aa)$.
  \item To every span $X \xleftarrow{f} Z \xrightarrow{g} Y$, we assign the
    functor $g_{!}f^{\ast} : \Aa_{/X} \to \Aa_{/Y}$. Note that this is vertical
    composition by the span.
  \item To every map $f : X \to Y$, we associate the profunctor $(\Aa_{/X})\op
    \times \Aa_{/Y} \to Set$ assigning $x : Z_1 \to X$ and $y : Z_2 \to Y$ to
    the set of commuting squares
    \[
      \left\{
\begin{tikzcd}
Z_1 \arrow[d, "x"'] \arrow[r, dashed] & Z_2 \arrow[d, "y"] \\
X \arrow[r, "f"']                     & Y                 
\end{tikzcd}
      \right\}
    \]
    acted upon by composition on the left and right. 
  \item To every map of spans
    \[
\begin{tikzcd}
X_1 \arrow[r]                                    & X_2                                    \\
Z_1 \arrow[r] \arrow[u, "f_1"] \arrow[d, "g_2"'] & Z_2 \arrow[u, "f_2"'] \arrow[d, "g_2"] \\
Y_1 \arrow[r]                                    & Y_2                                   
\end{tikzcd}
    \]
    associate the morphism of profunctors which acts by pulling back along the
    first square and pushing forward along the second.
\end{itemize}
\end{defn}

\begin{thm}\label{thm:representable.indexed.double.functor}
Let $\lens{u}{\id} : \lens{TS}{S} \leftrightarrows \lens{I}{S}$ be a $(\type{Bun},
T)$-dynamical system which exposes its entire state. Then we have an indexed double functor
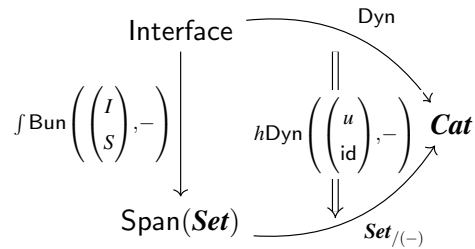
\[
\begin{tikzcd}
\type{Interface} \arrow[dd, "{\int \type{Bun}\left(\lens{I}{S}, -\right)}"'] \arrow[rrd, "\type{Dyn}", bend left] & {} \arrow[dd, "{h\type{Dyn}\left(\lens{u}{\id}, -\right)}" description, Rightarrow] &      \\
                                                                                                                &                                                                                   & \Cat \\
\type{Span}(\Set) \arrow[rru, bend right, "\Set_{/(-)}"']                                                                       & {}                                                                                &     
\end{tikzcd}
\]
covariantly represented by $\lens{u}{\id}$.
\end{thm}

The double functor $\int\type{Bun}\left( \lens{I}{S}, - \right) : \type{Interface}
\to \type{Span}(\Set)$ is given by interpreting $\type{Interface}$ as a double
category of certain spans in $\int \type{Bun}$ (by Proposition
\ref{prop:groth.double.span}) and noting that this representable functor
preserves pullbacks.

The transformation $h\type{Dyn}\left( \lens{u}{\id}, - \right) : \type{Dyn}
\Rightarrow \Set_{/(-)}$ consists of:
\begin{itemize}
\item For every interface $\lens{I'}{O'}$, we have a functor
  $\type{Dyn}\lens{I'}{O'} \to \Set_{ {\bigg/} \int\type{Bun}\left( \lens{I}{S},
      \lens{I'}{O'} \right)}$ sending a $\lens{I'}{O'}$-dynamical system
  $\lens{u'}{r'}$ to the set of covariant morphisms $h\type{Dyn}\left(
    \lens{u}{\id}, \lens{u'}{r'} \right)$ together with the projection of its
  first component to
  $\int\type{Bun}\left( \lens{I}{S}, \lens{I'}{O'} \right)$. This functor acts on
  morphisms by composition.
\item For every bundle map $\lens{g_{\sharp}}{g} : \lens{I'}{O'}
  \rightrightarrows \lens{I''}{O''}$, we have a map of profunctors given by
  post-composition:
  \[
\left\{ 
    \begin{tikzcd}
        \lens{TS'}{S'} \arrow[r, dashed, shift left,
        "\lens{T\varphi}{\varphi}"]\arrow[r, dashed, shift right] \arrow[d, leftarrow,  shift left] \arrow[d, shift right, "\lens{u'}{r'}"'] & \lens{TS''}{S''} \arrow[d, leftarrow, shift left, "\lens{u''}{r''}"] \arrow[d, shift right] \\
        \lens{I'}{O'}\arrow[r, shift right, "\lens{g_{\sharp}}{g}"']\arrow[r, shift left] & \lens{I''}{O''}
    \end{tikzcd}
 \right\} \to
\left\{ 
    \begin{tikzcd}
        h\type{Dyn}\left( \lens{u}{r}, \lens{u'}{r'} \right) \arrow[r, dashed, "\lens{T\varphi}{\varphi} \circ -"] \arrow[d] & h\type{Dyn}\left( \lens{u}{r}, \lens{u''}{r''} \right)  \arrow[d] \\
        \int\type{Bun}\left( \lens{I}{S}, \lens{I'}{O'} \right)\arrow[r] & \int\type{Bun}\left( \lens{I}{S}, \lens{I''}{O''} \right)
    \end{tikzcd}
 \right\}
  \]
\end{itemize}

In his behavioral approach to control theory \cites{Willems2007a}{Willems2007}, Willems
looks not at the dynamical law governing a system, but rather the trajectories
of the system and the variables that they expose. In their \emph{Temporal Type
  Theory} \cite{ShultzSpivak2019}, Schultz and Spivak give a definition of
\emph{behavior type} --- a sheaf on the interval domain of real numbers ---
which forms a foundation for a Willems-style analysis of various dynamical
systems. A Willems-style dynamical system is then a behavior type $S$ of
trajectories together with exposed variables $x : S \to V$ (landing in some
behavior type of values $V$). If $V \xleftarrow{s} X \xrightarrow{t} W$ is a
span representing a way that the variables of sort $V$ will be shared to form
variables of sort $W$, then the system resulting from the sharing of these variables is $t_{!}s^{\ast}x : S
\times_{V} X \to W$. We can express this idea through the indexed double
category $\Ba_{/(-)} : \type{Span}(\Ba) \to \Cat$ which assigns to each behavior
type $V$ of variables the category of systems exposing variables of sort $V$,
and which assigns to any span the variable sharing functor which it describes.

We will show that taking solutions of continuous-time dynamical systems
constitutes an indexed double functor $\type{Dyn} \to \Ba_{/(-)}$ taking a
dynamical system to its behavior type of trajectories. Indexed double
functoriality here shows that plugging in exposed variables to parameters can be
seen as an instance of variable sharing. In particular, it shows that if one
takes a system of systems of differential equations with parameters, solves them
in terms of those parameters, and then makes substitutions of those parameters
in terms of exposed variables of the systems, this is the same as substituting
and then solving. The proof relies crucially on the
representablility (by $\lens{\frac{d}{ds}}{\id} : \lens{T\Rb}{\Rb}
\leftrightarrows \lens{\Rb}{\Rb}$) of solutions of continuous-time dynamical
systems. We can see the following corollary as a paradigm for turning the open
dynamical system framework of free parameters (inputs) and exposed variables
(outputs) coupled by a dynamical law (the system itself) into a Willems-style
behavioral dynamical system consisting of trajectories and exposed variables.
\begin{cor}\label{thm:indexed.double.functor.behavior.type}
  Let $\Ba$ denote the category of Schultz-Spivak behavior types. There is an indexed double functor
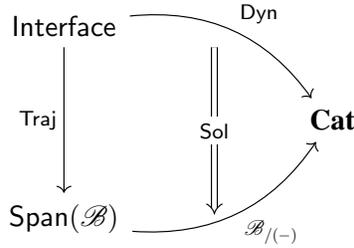
\[
\begin{tikzcd}
\type{Interface} \arrow[dd, "\type{Traj}"'] \arrow[rrd, "\type{Dyn}", bend left] & {} \arrow[dd, "\type{Sol}" description, Rightarrow] &      \\
                                                                                                                &                                                                                   & \Cat \\
\type{Span}(\Ba) \arrow[rru, bend right, "\Ba_{/(-)}"']                                                                       & {}                                                                                &     
\end{tikzcd}
\]
taking the behavior types of solutions of $(\type{Subm}, T)$-dynamical systems
--- that is, of continuous-time dynamical systems. 
\end{cor}
\begin{proof}[Proof Sketch]
We note that a solution of length $\ell$ is represented by the dynamical system
$$\lens{\frac{d}{ds}}{\id} : \lens{T(0, \ell)}{(0, \ell)} \leftrightarrows
\lens{(0, \ell)}{(0, \ell)}.$$
There is a functor $\Ib\Rb_{/\vartriangleright} \to h\type{Dyn}$ from
the site $\Ib\!\Rb_{/\vartriangleright}$ of $\Ba$ (this is the twisted arrow
category of the one object category associated to the additive monoid of
non-negative reals; see Definition 3.1 of \cite{ShultzSpivak2019}) to the
category of continuous-time dynamical systems and covariant morphisms sending
$\ell$ to the dynamical system $\lens{\frac{d}{ds}}{\id}$. Therefore, the various representable functors defined in Theorem
\ref{thm:representable.indexed.double.functor} take values in presheaves on $\Ib\!\Rb_{/\vartriangleright}$. We may then check that the values actually land in
sheaves by noting this property for the presheaf of solutions.
\end{proof}

\section{Conclusion}

In this extended abstract, we have laid out an abstract framework in which to
study open dynamical systems of many different kinds. We have analyzed morphisms
between systems into two sorts --- those which act on parameters
covariantly, and those which act on parameters contravariantly --- and presented
a double category of dynamical systems formed by these two sorts of morphisms.

We saw trajectories and steady states as an example of covariant morphisms,
and ``plugging in variables to parameters'' as an example of contravariant
morphisms. These are not the only uses of these morphisms. For example, in
the monadic doctrine of Markov decision processes, \emph{hierarchical
  planning} \cite{spivakmarkov} can be seen as an example of a contravariant
morphism, a perspective which the author looks to take in future work. 

In Theorem \ref{thm:representable.indexed.double.functor}, we showed that if one
takes a system of open dynamical systems, finds their trajectories, steady
states, or periodic orbits, and then makes substitutions of their parameters
in terms of the exposed variables of the systems, this is the same as substituting
and then finding those trajectories, steady states, or periodic orbits. These substitutions can be seen as a
form of ``matrix arithmetic'' for combining these covariantly representable
structures associated to open dynamical systems. 

This theorem suggests a possible way to speed up numerical approximation of
large systems with many repeated subparts. Namely, one approximates the solution
for the various sorts of subparts in terms of their parameters, and then
substitutes in those solutions according to the scheme by which the large system
was formed out of its subparts. Exploring the viability of this approach will be
the subject of future work.

\printbibliography

\end{document}